\documentclass[11pt, english]{article}
 \pdfoutput=1
\usepackage{hyperref}
\usepackage[utf8]{inputenc}
\usepackage{amsxtra}
\usepackage{color}
\usepackage{amsopn}
\usepackage{amsmath,amsthm,amssymb,amscd}
\usepackage{mathtools}
\usepackage{eufrak}
\usepackage{tikz}
\usepackage{tikz-cd}
\usepackage{calrsfs}
\usepackage[all,cmtip]{xy}
\usepackage{footmisc}

\theoremstyle{remark}
\newtheorem{remark}{Remark}[]

\theoremstyle{definition}

\newtheorem{rem}[remark]{Remark}

\theoremstyle{theorem}

\newtheorem{prop}[remark]{Proposition}
\newtheorem{corol}[remark]{Corollary}
\newtheorem{lemma}[remark]{Lemma}

\newcommand{\Gtwo}{G_2}

\newcommand{\beq}{\begin{equation}}
\newcommand{\eeq}{\end{equation}}
\newcommand{\bqn}{\begin{eqnarray}}
\newcommand{\eqn}{\end{eqnarray}}
\newcommand{\bqne}{\begin{eqnarray*}}
\newcommand{\eqne}{\end{eqnarray*}}

\newcommand{\R}{{\mathbb R}}

\newcommand{\Z}{\mathbb Z}

\newcommand{\SU}{{\rm SU}}
\newcommand{\SO}{{\rm SO}}

\newcommand{\RP}{\mathbb{RP}}

\newcommand{\Spin}{\mathrm{Spin}}

\newcommand{\PS}{\mathbb{P}(S)}
\usepackage[]{authblk}

\begin{document} 
\title{{A flow of isometric $\Gtwo$-structures. Short-time existence}}
\author{Leonardo Bagaglini}
\affil{Dipartimento di Matematica e Informatica {``}Ulisse Dini{"}. Università degli Studi di Firenze. E-mail address: leonardo.bagaglini@unifi.it.}
\maketitle
\abstract{We introduce a flow of $\Gtwo$-structures defining the same underlying Riemannian metric, whose stationary points are those structures with divergence-free torsion. We show short-time existence and uniqueness of the solution.}

\section*{Introduction}
Let $M$ be a seven dimensional, closed, spin manifold and let $\bar{\varphi}$ be the fundamental three-form of a $\Gtwo$-structure on $M$ with underlying metric $g$ and Hodge operator $\star$.
We study the following flow of $\Gtwo$-structures:
\begin{equation}\label{Gflow}
\begin{cases}
\dot{\varphi}_t=-2\star_t\left(\mathrm{div}(T^{\varphi_t})\wedge\varphi_t\right),\\
\varphi_0=\bar{\varphi},
\end{cases}
\end{equation}
where $\mathrm{div}(T^{\varphi})$ is the divergence of the full torsion tensor $T^\varphi$ associated to $\varphi$.\par
 It is evident that this flow preserves the metric defined by $\bar{\varphi}$; indeed the infinitesimal deformation of the three-form has no components in $\Omega^3_1(M)\oplus\Omega^3_{27}(M)$. Therefore any solution has to be sought in the class of $\Gtwo$-structures compatible with  $g$. This set, denoted by $[\bar{\varphi}]$, naturally inherits a differential structure; precisely it can be identified with the projective spinor bundle associated to $g$.\par
 We can therefore read the evolution equation firstly as a flow of projective spinors and secondly, thanks to a natural identification of $\R^7$ in the seven dimensional (irreducible) real spin module $\Delta$, as a flow of vector fields over $M$. The new flow, though non-linear, turns out to be parabolic and consequently to have short-time existence and uniqueness; indeed we prove the following proposition.
\begin{prop}\label{prop1}
Flow \eqref{Gflow} admits a unique solution $\varphi_t$ for short positive times.
\end{prop}\par
Moreover, as we might expect, we find out that the (negative) gradient flow of the energy functional
$$E:\,[\bar{\varphi}]\ni \varphi\mapsto\frac{1}{2}\int_M \star\left|T^\varphi\right|^2\in\R,$$
with respect to a natural $L^2$-metric, is exactly \eqref{Gflow}. We note that this functional is the same considered by Weiss and Witt in \cite{WW12} (with some changes due to other conventions). Anyway we consider its restriction to the subspace of isometric $\Gtwo$-structures; as consequence the related gradient flows are not equivalent though related.\par
Finally we compute the first and second variations of $E$ finding out that critical points, generically, are not local minima. Moreover we prove that the space of infinitesimal deformations of structures having divergence-free torsion is finite dimensional and we explicitly describe the obstruction space to actual deformations.\par\bigskip
The divergence-free condition seems to be a sort of fixing condition for $\Gtwo$-structures in a chosen isometry class. In \cite{Gri2} Grigorian interpreted this class as a class of connections on some `octonion' line bundle proving that condition $\mathrm{div}(T)=0$ corresponds to the Euler-Lagrange equation of an energy functional in this new setting. It is very likely that our point of view will be equivalent to the Grigorian's one, anyway, we use a purely spinorial approach which could light up some different aspects of the involved tensorial quantities.\par
More in depth one may attempt to fix the Laplacian co-flow of $\Gtwo$-structures, which is not well understood (differently from its counterpart), by adding, in some sense, the divergence-free condition. In fact the degeneracy of its evolution operator exactly sits in the infinitesimal directions belonging to $\Omega^4_7(M)$, which, in turn, do not perturb the metric (see \cite{Gri} for an exhaustive treatment). \par\bigskip
Throughout the various sections we adopt the Einstein notation. When a Riemannian metric is fixed we always identify vector fields and one-forms via the Riesz isomorphism.  
\section{$\Gtwo$-structures}
Let $M$ be a seven dimensional manifold. As usual we denote by $\Gamma(M,W)$ the space of smooth sections of a vector bundle, of total space $W$, over $M$.\par
A three-form $\varphi\in\Omega^3(M)$ on a seven dimensional manifold is said to be \emph{stable} if the $\Omega^7(M)$-valued symmetric two-form $b_\varphi$ on $TM$, the tangent bundle of $M$, given by
$$6\,b_\varphi\left(X,Y\right)=\left(X\neg\varphi\right) \wedge \left(Y\neg\varphi\right)\wedge \varphi,\quad X,Y\in\Gamma(M,TM),$$
is non-degenerate. If so $\varphi$ defines a non-vanishing seven-form $\varepsilon_\varphi$ on $M$ by $\sqrt[9]{\mathrm{det}(b_\varphi)}$. A stable form $\varphi$ is said to be \emph{positive} if the metric $g_\varphi$ defined by
$$b_\varphi=g_\varphi\varepsilon_\varphi,$$
is positive definite.\par
A positive three-form of constant length $|\varphi|_\varphi=\sqrt{7}$ is said to be the \emph{fundamental three-form} of a $\Gtwo$-structure. Indeed it can be shown that $\varphi$ uniquely defines a $\Gtwo$-structure, in the sense of principal bundles, on $M$; precisely the stabiliser, under the natural $\mathrm{GL}(T_pM)$ action, of $\varphi_p$ is isomorphic to $\Gtwo$ for any $p\in M$. In this case the formula
$$\varphi(X,Y,Z)=g_\varphi(X\times Y,Z),\quad X,Y,Z\in\Gamma(M,TM),$$
defines a vector cross product $\times$ on $TM$ which is bilinear and skew-symmetric. We denote by $\mathcal{G}$ the space of $\Gtwo$-structures on $M$. 
\begin{rem}
It is a classical result that $M$ admits $\Gtwo$-structure if and only if it is spin (see for instance \cite{Bry}).
\end{rem} 
\par 
Assume that $M$ comes with a Riemannian metric $g$ and that $\varphi$ is the fundamental three-form of a $\Gtwo$-structure. Then the structure, or $\varphi$, is said to be \emph{compatible} with $g$ if and only if $g=g_\varphi$.\par
Given a fundamental three-form $\varphi$ of a $\Gtwo$-structure over $M$ there exists a $(0,2)$-tensor $T$, called the \emph{full torsion tensor} of $\varphi$ (with different notations see \cite{Bry}, \cite{Gri} and many others), defined as
$$\nabla^{\phantom{a}}_a\varphi^{\phantom{\varphi}}_{bcd}=2T_{a}^{\phantom{e}e}(\star_\varphi\varphi)_{ebcd}^{\phantom{a}},$$
where $\nabla$ is the Levi-Civita connection of $g_\varphi$ and $\star_\varphi$ denotes the Hodge operator of $g_\varphi$ and $\varepsilon_\varphi$. Finally we define the \emph{divergence} of $T$ as the one-form
$$(\mathrm{div}\, T)_b=\nabla^a T_{a b}.$$

\section{$\Gtwo$-structures with the same Riemannian metric}
Let $(M,g)$ be a closed seven dimensional Riemannian spin manifold and let $\bar{\varphi}$ be the fundamental three-form of a $\Gtwo$-structure on $M$ compatible with $g$. Denote by $F$, $\nabla$ and $\star$ the special orthogonal frame bundle, the Levi-Civita connection and the Hodge operator of $g$ respectively, and by $[\bar{\varphi}]$ the class of $\Gtwo$-structures compatible with $g$.\par
A convenient way to handle the class $[\bar{\varphi}]$ is by considering the eight dimensional (irreducible) real spin module $\Delta$ (for other approaches see the recent \cite{Gri2}) Let $\mathcal{F}$ be the total space of a spin structure on $(M,g)$ and $S$ be the spinor bundle associated to $g$, that is $S=\mathcal{F}\times_{\Spin(7)}\Delta$.\par
Then $\bar{\varphi}$ defines a unique (up to sign) unit spinor field $\bar{\psi}\in\Gamma(M,S)$ by
\begin{equation}\label{spinform}
\bar{\varphi}(X,Y,Z)=(X\cdot Y\cdot Z\cdot \bar{\psi},\bar{\psi}),\quad X,Y,Z\in TM,
\end{equation}
where the brackets denote the (real) spin metric on $S$ and $\cdot$ denote the Clifford multiplication (see for instance \cite{Chi}). The following Lemma contains some useful formulas involving the Clifford multiplication and the Hodge operator.
\begin{lemma}[\cite{Chi}]
For any vector fields $X,Y,Z$ and $W$ on $M$ the followings hold:
\begin{enumerate}\label{cross}
\item $X\cdot Y\cdot \bar{\psi}=-\left(X\bar{\times} Y\right)\cdot\bar{\psi}-g(X,Y)\bar{\psi}$;
\item $(\star\bar{\varphi})(X,Y,Z,W)=\bar{\varphi}(X,Y,Z\bar{\times}W)-g(X,Z)g(Y,W)+g(X,W)g(Y,Z)$. 
\end{enumerate}
\end{lemma}\par
Observe that, for any vector field $U$ and function $u$ on $M$ satisfying $|U|^2+|u|^2=1$, the spinor field
\begin{equation}\label{spinor}
\psi_{U,u}=U\cdot\bar{\psi}+u\bar{\psi},
\end{equation}
has unit length, indeed the Clifford multiplication gives an isometry between $T_pM$ and $\bar{\psi}_p^\perp$ for each $p\in M$; this is an isomorphism of $\Gtwo$-modules. Moreover Formula \eqref{spinform} defines a $1:1$ correspondence between projective spinor fields and elements in $[\bar{\varphi}]$:
\begin{equation}\label{iso}
\begin{CD}
\Gamma(M,\mathbb{P}(S))@>\Phi>>[\bar{\varphi}],\\
[\psi]@>>> \Phi(\psi),
\end{CD}
\end{equation}
which is modelled on the presentation of $\RP^7=\mathbb{P}(\Delta)$ as $\SO(7)/\Gtwo=\Spin(7)/\left(\Gtwo\times\Z_2\right)$. These two observations, combining together, give an explicit description, in terms of three-forms, of $\Phi$ involving vector fields only. It is given in the following Lemma (formula (3.5) in \cite{Bry}) which, for consistence, we are going to prove.
\begin{lemma}[\cite{Bry}]\label{psiform}
The fundamental three-form of the $\Gtwo$-structure $\Phi\left(\left[\psi_{U,u}\right]\right)$ is explicitly given by
\begin{align*}
\Phi\left(\left[\psi_{U,u}\right]\right)=(u^2-|U|^2)\bar{\varphi}+2u\star(\bar{\varphi}\wedge U)+2\left({U}\neg\bar{\varphi}\right)\wedge U.
\end{align*}
\end{lemma}
\begin{proof}
We have to show that, for any vector fields $X,Y,Z$, the function
$$\left(X\cdot Y\cdot Z\cdot \psi_{U,u} ,\psi_{U,u}\right),$$
has the claimed expression.
First let us observe that
\begin{align*}
u^2\left(X\cdot Y\cdot Z\cdot \bar{\psi},\bar{\psi}\right)=u^2\bar{\varphi}(X,Y,Z).
\end{align*}
Similarly, using the skew-symmetry of the Clifford multiplication the the Clifford identity $A\cdot B+B\cdot A =-2g(A,B)$ for arbitrary vector fields $A,B$, we obtain
\begin{align*}
\left(X\cdot Y\cdot Z\cdot U\cdot \bar{\psi},\bar{\psi}\right)=&
-\left(U\cdot X\cdot Y\cdot Z\cdot U\cdot \bar{\psi},\bar{\psi}\right)\\
=&2g(U,X)\left( Y\cdot Z\cdot U\cdot \bar{\psi},\bar{\psi}\right)+\left(X\cdot U\cdot Y\cdot Z\cdot U\cdot \bar{\psi},\bar{\psi}\right)\\
=&2g(U,X)\bar{\varphi}(Y,Z,U)+\left(X\cdot U\cdot Y\cdot Z\cdot U\cdot \bar{\psi},\bar{\psi}\right)\\
=&\dots\\
=&2\left(U\wedge\left(U\neg\bar{\varphi}\right)\right)(X,Y,Z)-|U|^2\bar{\varphi}(X,Y,Z).
\end{align*}
Finally let us observe that, as previously seen,
\begin{align*}
&u\left(X\cdot Y\cdot Z\cdot U\cdot \bar{\psi},\bar{\psi}\right)+u\left(X\cdot Y\cdot Z\cdot \bar{\psi},U\cdot\bar{\psi}\right)=\\
&2u\left(X\cdot Y\cdot Z\cdot U\cdot \bar{\psi},\bar{\psi}\right)-2u g(U,X)g(Y,Z)+2ug(U,Y)g(X,Z)-2ug(U,Z)g(X,Y).
\end{align*}
But, thanks to Lemma \ref{cross}, the first term in the right side becomes $-2u\left(X\cdot Y\cdot (Z\bar{\times} U)\cdot\bar{\psi},\bar{\psi}\right)+2ug(X,Y)g(U,Z)$, leaving
$$2u\left[-\bar{\varphi}(X,Y,Z\bar{\times}U)+g(X,Z)g(Y,U)-g(X,U)g(Y,Z)\right],$$
which is equal to $-2u(\star\varphi)(X,Y,Z,U)=2u(U\neg(\star\varphi))(X,Y,Z)$ again by Lemma \ref{cross}.
Therefore, since $U\neg\star \varphi=\star(\varphi\wedge U)$, the lemma follows by putting all together.
\end{proof}
\par
The full torsion tensor, or intrinsic endomorphism, $\bar{T}$ of $\bar{\varphi}$ is related to $\bar{\psi}$ as follows (see \cite{Chi})
\begin{equation}\label{torsion}
\left(\nabla_X\bar{\psi}|Y\cdot\bar{\psi}\right)=\left(\bar{T}(X)\cdot\bar{\psi}|Y\cdot\bar{\psi}\right)=\bar{T}(X,Y),\quad X,Y\in TM.
\end{equation}\par
In the light of Lemma \ref{cross} and Formula \eqref{spinor} we can derive the explicit expression of the full torsion of a generic $\varphi\in[\bar{\varphi}]$.
\begin{lemma}\label{lemtor}
Let $\varphi\in[\bar{\varphi}]$. Then the full torsion tensor $T$ of $\varphi$ is related to $\bar{T}$ by the following formula
\begin{align}
T(X,Y)=& u\, g\left(\nabla_X U,Y\right)-(\nabla_Xu)g(U,Y)+\bar{\varphi}(\nabla_XU,U,Y)\\
\nonumber&(u^2-|U|^2)\bar{T}(X,Y)+2\bar{T}(X,U)g(U,Y)-2u\,\bar{\varphi}(U,\bar{T}(X),Y),
\end{align}
where $X,Y\in\Gamma(M,TM)$ and $\varphi=\Phi(\psi_{U,u})$.
\end{lemma}
\begin{proof}
Let $X,Y$ be two vector fields over $M$. By definition we have to compute $T(X,Y)=(\nabla_X\psi_U,Y\cdot\psi_U).$ First we observe that, by \eqref{spinor},
$$\nabla_X\psi_U=\nabla_X(U\cdot\bar{\psi}+u\bar{\psi})=(\nabla_XU)\cdot\bar{\psi}+U\cdot \bar{T}(X)\cdot\bar{\psi}+(\nabla_Xu)\bar{{\psi}}+u \bar{T}(X)\cdot\bar{\psi}.$$
Then, being $U\cdot \bar{T}(X)\cdot\bar{\psi}=-(U\bar{\times}\bar{T}(X))\cdot\bar{\psi}-\bar{T}(X,U)\bar{\psi}$ by Lemma \ref{cross}, the previous becomes
$$\nabla_X\psi_U=\left(\nabla_XU-U\bar{\times}\bar{T}(X)+u \bar{T}(X)\right)\cdot\bar{\psi}+
\left(-\bar{T}(X,U)+\nabla_Xu\right)\bar{\psi}.
$$
Similarly we obtain
$$Y\cdot\psi_U=\left(uY-Y\bar{\times}U\right)\cdot\bar{\psi}-g\left(Y,U\right)\bar{\psi}.$$
Finally putting all together and using Lemma \ref{cross} we derive the claimed formula. 
\end{proof}
\section{The flow}
We are going to study the following geometric flow 
\begin{equation}\label{gflow}
\begin{cases}
\dot{\varphi}_t=-2\star_t\left(\mathrm{div}(T^{\varphi_t})\wedge\varphi_t\right),\\
\varphi_t\in[\bar{\varphi}],\\
\varphi_0=\bar{\varphi},
\end{cases}
\end{equation}
where $\mathrm{div}(T^\varphi)$ denotes the divergence of the full torsion tensor associated to $\varphi$.
\begin{rem}
Notice that if $\varphi$ satisfies the evolution equation in \eqref{gflow} then it is compatible with $g$. Indeed, by Proposition 4 in \cite{Bry}, $\dot{g}$ identically vanishes. Therefore condition $\varphi_t\in[\bar{\varphi}]$ is automatically satisfied and $\star_t=\star$.
\end{rem}\par
Let us prove the following Lemma.
\begin{lemma}\label{zariskispin}
The Zariski tangent space to $\Gamma(M,\PS)$ at a projective spinor field $[\psi]$ is naturally identified with $\Gamma(M,TM)$ via the Clifford multiplication. 
\end{lemma}
\begin{proof}
The proof is a clear consequence of Formula \eqref{spinor}. Indeed if $\psi$ is a unit spinor field then we can parametrize nearby unit spinor fields by vector fields. Hence any variation $\delta{\psi}$ of $\psi$ will be given by $\delta{U}\cdot \psi$ for some vector field $\delta{U}$ on $M$, being $\delta{u}=0$ in $\psi$. Then the lemma follows by considering the natural projection of $S$ on $\PS$.
\end{proof}
We want to reformulate \eqref{gflow} as a flow in $\Gamma(M,\PS)$ through \eqref{iso}:
\begin{equation}\label{projflow}
\begin{cases}
\dot{\left[\psi_t\right]}=P(\left[\psi_t\right]),\\
[\psi_0]=[\bar{\psi}],
\end{cases}
\end{equation}
where $P$ will be an operator with values in the space of sections of $V\PS$, the vertical bundle of $\PS$, over $M$.
To this aim let us consider the first variation of $\Phi$ at a generic $[{\psi}]$. Then, since $\dot{\psi}=\dot{U}\cdot \psi$ for some vector field $\dot{U}$ as showed in Lemma \ref{zariskispin}, we find out that
$$\dot{\varphi}=D\Phi_{[{\psi}]}[\dot{\psi}]=2\star({\varphi}\wedge \dot{U}).$$
Therefore flow \eqref{gflow} and flow \eqref{projflow} are equivalent if we choose $P$ as
\begin{equation}\label{P}
P([\psi])=\mathrm{div}\,T^{\Phi([\psi])}\cdot\psi.
\end{equation}
\section{Short-times existence and uniqueness}\label{steu}
In this section we show that Equation \eqref{projflow} admits a unique solution. Since we are interested in finding solutions for short times we can consider the lifted flow acting on $\Gamma(M,S)$. In this space, by \eqref{spinor}, any solution will be given by $\psi_{U_t,u_t}$, or  simply $\psi_{U_t}$, for a time-dependent vector field $U_t$ with $|U_t|<<1$ and $u_t=\sqrt{1-|U_t|^2}$.
\begin{lemma}
For any solution $\psi_{U_t}$ of \eqref{projflow} $U_t$ satisfies
\begin{equation}\label{vectorfieldflow}
\begin{cases}
\dot{U}_t=-\mathrm{div}\, T^{U_t}\bar{\times}U_t+u_t\,\mathrm{div}\, T^{U_t},\\
U_0=0,
\end{cases}
\end{equation} 
where $\mathrm{div}\,T^{U_t}$ is nothing else than the divergence of the full torsion tensor of $\psi_{U_t}$, and the vice-versa holds.
\end{lemma}
\begin{proof}
Suppose flow \eqref{projflow} has a solution. Then, by \eqref{P}, it is
$$\frac{d}{dt}\psi_{U}=\mathrm{div}\, T^{U}\cdot\psi_{U}.$$
Expanding this expression gives
\begin{align*}
\dot{U}\cdot \bar{\psi}+\dot{u}\bar{\psi}=\left(-\mathrm{div}\, T^{U}\bar{\times}U+u \mathrm{div}\, T^{U}\right)\cdot\bar{\psi}-g\left(\mathrm{div}\, T^{U},U\right)\bar{\psi}.
\end{align*}
Therefore if $U$ solves \eqref{vectorfieldflow} then the first of these two equations is satisfied. But, since $\dot{u}=u^{-1}g(U,\dot{U})$, the second follows as well. The other direction is also clear.
\end{proof}
\begin{rem}\label{remspinflow}
It is clear that any solution of \eqref{projflow} can be lifted to a family of (unit) spinor fields satisfying the following, equivalent, evolution equation
\begin{equation}\label{spinflow}
\begin{cases}
\dot{\psi}_t=\mathrm{div}\, T^{\psi_t}\cdot\psi_t,\\
{\psi}_0=\bar{\psi},
\end{cases}
\end{equation}
where, clearly, $T^\psi=T^{\Phi([\psi])}$.\par
This flow is manifestly strongly elliptic, in fact the principal symbol of the linearized evolution operator at a co-vector $\xi$ is the multiplication by $|\xi|^2|\psi|^2$,
and therefore we could prove it has short-time existence and uniqueness; of course $|{\psi}_t|$ will be constant. However, for a future reference, we are interested in to understand the evolution operator in \eqref{vectorfieldflow}. Therefore we will proceed in this direction.
\end{rem}
At this point, in order to prove Proposition \ref{prop1}, we have only to show that \eqref{vectorfieldflow} has short-time existence and uniqueness. But first let us prove the following lemma. 
\begin{lemma}\label{ellipticity}
There exists an open neighbourhood $\mathcal{U}$ of the zero section in $\Gamma(M,TM)$, in the $\mathcal{C}^0$-topology, such that the non-linear order $2$ differential operators $P''$ and $P'$
\begin{gather*}
P''\,:\mathcal{U}\ni U\mapsto \mathrm{div}\, T^{U}\in\Gamma(M,TM),\\
P'\,:\mathcal{U}\ni U\mapsto -\mathrm{div}\, T^{U}\bar{\times}U+u\,\mathrm{div}\, T^{U}\in\Gamma(M,TM),\end{gather*}
are strongly elliptic.
\end{lemma}
\begin{proof}
Let $\mathcal{U}$ be the open set of vector fields with length less than $1$; to avoid the singularity $\sqrt{1-|U|^2}$. Let $U\in\mathcal{U}$ and $Q_U$ be the linearization of $P''$ at $U$, that is $DP''_U$. For simplicity we write $Q=Q_U$. A straightforward computation shows that, by Lemma \ref{torsion},
$$Q(V)=u\nabla^a\nabla_{a}V_{b}-\left[D\left(\nabla^a\nabla_{a} u \right)\right]_U(V)U_{b}+(\nabla^a\nabla_{a}V^m)U^n\bar{\varphi}_{bmn}+l.o.t.,$$
for any vector field $V$. Keeping in mind that $u=\sqrt{1-|U^2|}$ we derive that $\left(\nabla^a\nabla_{a} u \right)=-\frac{U^m(\nabla^a\nabla_{a}U^n)g_{mn}}{\sqrt{1-|U|^2}}+l.o.t.$, which leads to
$$\left[D\left(\nabla^a\nabla_{a} u \right)\right]_U(V)=-\frac{U^m(\nabla^a\nabla_{a}V^n)g_{mn}}{\sqrt{1-|U|^2}}+l.o.t.$$
Putting all together we obtain
\begin{align*}
&Q(V)=u\nabla^a\nabla_{a}V_{b}+\frac{1}{u}\left[{U^m(\nabla^a\nabla_{a}V^n)g_{mn}} \right]U_{b}+(\nabla^a\nabla_{a}V^m)U^n\bar{\varphi}_{bmn}+l.o.t.
\end{align*}\par
Now let $\sigma_\xi Q$ be the principal symbol of $Q$ at a tangent vector $\xi$. It turns out that
\begin{align*}
(\sigma_\xi Q)(V)=(1-|U|^2)^{1/2}|\xi|^2V+(1-|U|^2)^{-1/2}|\xi|^2g(U,V)U+|\xi|^2V\bar{\times}U.
\end{align*}
In particular we see that 
$$g((\sigma_\xi Q)(V),V)=(1-|U|^2)^{1/2}|\xi|^2|V|^2+(1-|U|^2)^{-1/2}|\xi|^2g(U,V)^2.$$ Therefore $Q_U$ is a strongly elliptic linear differential operator of order $2$ away from $|U|=1$. \par
Finally, to prove that $P'$ is strongly elliptic, we observe that the extra term of $\sigma_\xi DP'_U(V)$ is given by
$$-(1-|U|^2)^{1/2}|\xi|^2V\bar{\times}U-|\xi|^2(V\bar{\times}U)\bar{\times}U.$$
Then, when paired with $V$, it gives $|\xi|^2|V\bar{\times}U|^2$; indeed $g((V\bar{\times}U)\bar{\times}U,V)=-|V\bar{\times}U|^2$ by the symmetries of $\bar{\times}$.
\end{proof}
\begin{prop}
Flow \eqref{vectorfieldflow} has unique solution for short positive times.
\end{prop}
\begin{proof}
In the light of Lemma \ref{ellipticity} the seasoned reader should argue that the Nash-Moser theorem  (see \cite{Ham1}) can be applied. Indeed this case turns to be extremely easy to treat, since the operator $P'$ is strongly elliptic in any direction; whilst in \cite{Ham2}, \cite{BX}, \cite{Gri} or \cite{Vez} this condition is satisfied in certain directions only.\par
However we avoid to repeat once more the celebrated argument. Instead we directly choose to apply Theorem 3.2 in \cite{Vez} in a very straightforward case. Indeed, with the same terminology and notation adopted there, it is enough to consider the trivial Hodge system $(TM,TM,\mathrm{Id}_{TM},\mathrm{Id}_{TM})$ and choose $L=P'$, $\tilde{L}_U=l_U=(DP')_U$ for any $U\in\mathcal{U}$.
\end{proof}
\section{Energy}
Let $S_1$ be the unit sphere subbundle of $S$. As we already noticed in Remark \ref{remspinflow}, Flow \eqref{Gflow} turns to be equivalent to \eqref{spinflow} in $\Gamma(M,S_1)$. This last can be easily read as the negative gradient flow of the functional $E$
\begin{equation}\label{energy}
E:\,\Gamma(M,S_1)\ni\psi\mapsto\frac{1}{2}\int_M\star\left|T^\psi\right|^2\in\R,
\end{equation}
as we will see in a moment. Recall that for any spinor field $\psi\in\Gamma(M,S_1)$  its Zariski tangent space can be identified with $\Gamma(M,TM)$ by using the Clifford multiplication (Lemma \eqref{zariskispin}). 
\begin{prop}\label{propenergy}
Flow \eqref{spinflow} is the gradient flow of $E$ with respect to the $L^2$-metric inherited by $\Gamma(M,TM)$.
\end{prop}
\begin{proof}
Let $\psi$ any unit spinor field. Any variation of $\psi$ in $\Gamma(M,S_1)$ will be given by $\delta \psi=U\cdot \psi$, for some vector field $U$, due to Lemma \eqref{zariskispin}. Let $X,Y$ be vector fields on $M$. Then the variation of $T^\psi$, by \eqref{torsion}, will be
\begin{align*}
\delta T^\psi(X,Y)=&\left(\nabla_X\delta\psi,Y\cdot\psi\right)+\left(\nabla_X\psi,Y\cdot\delta\psi\right)=\left(\nabla_X(U\cdot\psi),Y\cdot\psi\right)+\left(\nabla_X\psi,Y\cdot U\cdot \psi\right)\\
=&\left(\nabla_X(U)\cdot\psi,Y\cdot\psi\right)+\left(U\cdot(\nabla_X\cdot\psi),Y\cdot\psi\right)+\left(\nabla_X\psi,Y\cdot U\cdot \psi\right)
\end{align*}
From Lemma \ref{cross} we derive that 
\begin{align*}
\left(U\cdot(\nabla_X\cdot\psi),Y\cdot\psi\right)+\left(\nabla_X\psi,Y\cdot U\cdot \psi\right)&=-g\left(U\times T(X),Y\right)-g\left(T(X),Y\times U\right)\\
&=-2\varphi(T(X),Y,U),
\end{align*}
which leads to
$$\delta T^\psi(X,Y)=g\left(\nabla_X U,Y\right)-2\varphi(T(X),Y,U).$$
In local coordinates this reads as $(\delta T^\psi)^{\phantom{a}}_{ab}=\nabla^{\phantom{a}}_a U^{\phantom{a}}_b-2U^mT_a^{\phantom{m}n}\varphi^{\phantom{m}}_{mnb}$. Thus
\begin{align}\label{gT}
g(T^\psi,\delta T^\psi)&=T_{ab}\nabla^aU^b-2U^mT_a^{\phantom{a}n}\varphi^{\phantom{m}}_{mnb}\\
\nonumber &=-\nabla^a(T_{ab})U^b+\nabla^a(T_{ab}U^b)-2U^mT_a^{\phantom{a}n}\varphi^{\phantom{A}}_{mnb}.
\end{align}
\par
On the other hand the linearization of $E$ at $\psi$ is
\begin{equation*}
DE_\psi(U)=\int_M  g(T^\psi,\delta T^\psi)\star 1,
\end{equation*}
therefore, taking the integral of \eqref{gT}, we find out that
$$DE_\psi(U)=-\int_M g(\mathrm{div}\,T^\psi,U)+ 2U^mT^{ab}T_a^{\phantom{a}n}\varphi^{\phantom{m}}_{mnb}\star 1.$$
But $U^mT^{ab}T_a^{\phantom{a}n}\varphi^{\phantom{m}}_{mnb}$ identically vanishes since $T^{ab}T_a^{\phantom{a}n}$ is symmetric in $(b,n)$ and so $T^{ab}T_a^{\phantom{a}n}\varphi^{\phantom{m}}_{mnb}$ has to be zero for any $m$. 
\end{proof}
\begin{prop}\label{variations}
The first variation of $E$ at a unit spinor field $\psi$ is given by
\begin{align*}
&DE_\varphi(U)=-\int_M g(U,\mathrm{div}\,T^\varphi)\star 1=\int_Mg(\nabla U,T^\varphi)\star 1.
\end{align*}
Moreover, if $\varphi$ is a critical point of $E$ then its second variation at $\varphi$ is the symmetric bilinear form
\begin{align*}
&D^2 E_\varphi(U,V)=\int_M g\left(U, \Delta' V \right)+2g\left(U,(\nabla^aV^m)T_a^{\phantom{m}n}\varphi^{\phantom{m}}_{mnb}\right)\star 1,
\end{align*}
where $\Delta'$ denotes the Bochner Laplacian $-\nabla^a\nabla_a$ acting on vector fields and $U,V$ denote arbitrary vector fields.
\end{prop}
\begin{proof}
In Proposition \ref{propenergy} we have already computed the linearization of $E$. So let us assume $\mathrm{div}\, T^\varphi=0$ and consider the second derivative $D^2E_\psi(U,V)=D(DE_\cdot(U))_\psi(V)$ at arbitrary vector fields $U,V$. Then, since $(D T^\cdot)_\psi(V)_{ab}=\nabla_aV_b -2V^mT_a^{\phantom{a}n}\varphi^{\phantom{a}}_{mnb}$, it turns out that
$$g\left(U,(D\mathrm{div}\,T^\cdot)_\psi(V)\right)=U^b \nabla^a\nabla_aV_b-2U^b \nabla^a\left(V^mT_a^{\phantom{a}n}\varphi^{\phantom{a}}_{mnb}\right).$$
Finally $\nabla^a\left(V^mT_a^{\phantom{a}n}\varphi^{\phantom{m}}_{mnb}\right)=\nabla^a\left(V^m\right)T_a^{\phantom{a}n}\varphi^{\phantom{m}}_{mnb}$. Indeed 
$$\nabla^a\left(T_a^{\phantom{a}n}\varphi^{\phantom{m}}_{mnb}\right)=T_a^{\phantom{a}n}T^{ap}_{\phantom{p}}(\star\varphi)^{\phantom{p}}_{pmnb},$$ which identically vanishes being $T_a^{\phantom{a}n}T^{ap}_{\phantom{p}}$ symmetric in $(n,p)$. The claimed formula then follows.\end{proof}
Let $\varphi\in[\bar{\varphi}]$ and denote by $R^\varphi$ the operator
$$R^\varphi:\Gamma(M,TM)\ni U\mapsto 2\nabla^a\left(U^m\right)T_a^{\phantom{a}n}\varphi^{\phantom{mn}b}_{mn}\in\Gamma(M,TM).$$
\begin{corol}\label{1jet}
Let $\bar{\varphi}$ satisfy $\mathrm{div}\,T^{\bar{\varphi}}=0$. Then the moduli space of infinitesimal deformations of $\bar{\varphi}$ in $[\bar{\varphi}]$ having divergence-free torsion are finite dimensional, with dimension grater than or equal to $\mathrm{dim}\,\mathrm{Ker}(\Delta')$. Moreover the obstruction space to actual deformations of $\varphi$ is represented by
$$\mathrm{Ker}(\Delta'+R^{\bar{\varphi}})/\mathrm{Ker}(\Delta').$$  
\end{corol}
\begin{proof}
Recall that, as we showed in propositions \ref{propenergy} and \ref{variations}, the linearization of $\mathrm{div}\,T^{\bar{\varphi}}$ at $\bar{\varphi}$  is simply $-\Delta'-R^{\bar{\varphi}}$. In particular it is a strongly elliptic differential operator  which is also self-adjoint with respect to the standard $L^2$-metric and satisfies $\mathrm{Ker}(\Delta'+R^{\bar{\varphi}})\supseteq\mathrm{Ker}(\Delta')$.\par
On the other hand local deformations of $\bar{\varphi}$ having divergence-free torsion correspond to the zero set of $P''$, defined in \ref{ellipticity}. Actually the range of this map is in the $L^2$-orthogonal complement of $\mathrm{Ker}(\Delta')$, which is $\mathrm{Im}(\Delta')$ by self-adjointness. Indeed if $W$ belongs to this kernel it is covariantly constant, therefore
$$\left<\mathrm{div}\, T^U,W\right>_{L^2}=-\left< T^U,\nabla W\right>_{L^2}=0.$$\par
In particular infinitesimal deformations of $\bar{\varphi}$ belong to $\mathrm{Ker}(\Delta'+R^{\bar{\varphi}})$. Moreover if $\mathrm{Im}(\Delta')$ was equal to $\mathrm{Im}(\Delta'+R^{\bar{\varphi}})$ then, by the Banach space implicit function theorem, $\left\{P''=0\right\}$ would be a smooth manifold of dimension $\mathrm{dim}\,\mathrm{Ker}(\Delta')$ near $0$. Therefore the obstruction space turns to be $\mathrm{Im}(\Delta')/\mathrm{Im}(\Delta'+R^{\bar{\varphi}})$ which is equal to $\mathrm{Ker}(\Delta'+R^{\bar{\varphi}})/\mathrm{Ker}(\Delta')$.
\end{proof}
\begin{corol}
If $T^{\bar{\varphi}}=0$ then deformations of $\bar{\varphi}$ in $[\bar{\varphi}]$ satisfying $\mathrm{div}(T^\varphi)=0$ are unobstructed, torsion-free, and define a smooth manifold of dimension $b_1(M)$.
\end{corol}

\section*{Acknowledgements}
The author is grateful to Prof. Anna Maria Fino for her constant interest in his work, to Sergey Grigorian and Luigi Vezzoni for useful improvements, to Francesco Pediconi and Alberto Raffero for several pleasant conversations, and to ‘‘Università degli Studi di Firenze" for all the support he received.

\end{document}